\documentclass[12pt]{article}
\usepackage[english]{babel}

\usepackage[letterpaper]{geometry}
\usepackage{pdfpages}
\usepackage{geometry} 
\usepackage{amsmath}  
\usepackage{graphicx} 
\usepackage{float}
\usepackage{color}
\usepackage{amsmath}
\usepackage{amssymb}
\usepackage{amsthm}
\usepackage{array}
\usepackage{xcolor}
\usepackage{float}
\usepackage{bm}

\usepackage{amsmath}
\usepackage{graphicx}
\usepackage[colorlinks=true, allcolors=blue]{hyperref}

\newtheorem{theorem}{Theorem}

\theoremstyle{definition}

\theoremstyle{lemma}

\theoremstyle{remark}

\theoremstyle{boldremark}

\newtheorem{corollary}[theorem]{Corollary}

\title{Revisit escape path for infinite unit strip forest and unit broadworm}
\author{Zhipeng Deng}
\date{} 
\begin{document}
\maketitle

\begin{abstract}
Building on our previous general computational solution to Bellman’s Lost-in-a-Forest Problem, we present a new approach and analytical formulas for the previously well-known escape path for the infinite unit-strip forest and unit broadworm by Zalgaller. Earlier studies addressed these problems exclusively through geometric methods. We reformulated the problem as an interval-cover problem and then formulated it as a constrained functional minimization problem. This constrained functional minimization can be directly discretized and subsequently solved as a convex optimization. Furthermore, we extend the analysis of various line segment. Finally, we show that, in the case of a closed escape path for the unit strip, the optimal solution is a curve of constant unit width.
\end{abstract}

\\

\noindent \textbf{Classification}

Optimization and Control (math.OC)

Metric Geometry (math.MG)

Discrete Mathematics (cs.DM)

Computational Geometry (cs.CG)

49K30 (Optimal Solutions in Calculus of Variations)

49Q10 (Optimization of Shapes Other Than Minimal Surfaces)

52A40 (Inequalities and Extremum Problems)

\tableofcontents

\section{Introduction}
Bellman’s Lost-in-a-Forest Problem is a challenging unsolved minimization problem introduced by Richard E. Bellman \cite{Finch2004} \cite{Ward2008}. It is typically assumed that the hiker knows neither the starting location nor the initial facing orientation.

In the earliest forms, Bellman’s problem sought the shortest escape path from an infinite unit strip, defined as the region between two parallel lines. A solution was described by Zalgaller in 1961 \cite{Zalgaller2005}. He showed that the shortest escape path is a symmetric curve consisting of four line segments and two circular arcs. Its length is given by the following analytical expression \cite{Finch2004}:

\begin{equation}
\varphi = \arcsin \left( \frac{1}{6} + \frac{4}{3} \sin \left( \frac{1}{3} \arcsin \frac{17}{64} \right) \right),
\end{equation}
\begin{equation}
\psi = \arctan \left( \frac{1}{2} \sec \varphi \right).
\end{equation}
\begin{equation}
\ell_0 = 2 \left( \frac{\pi}{2} - \varphi - 2\psi + \tan \varphi + \tan \psi \right) \approx 2.278292
\end{equation}

This solution is highly non-trivial because it is non-smooth and piecewise defined. It has been discussed and analyzed in numerous papers \cite{Finch2004-1} \cite{Panraksa2007} \cite{Panraksa2007-1} \cite{Kübel2021} \cite{Movshovich2025} and is also well-known as the unit broadworm \cite{Gibbs2016} \cite{Khandhawit2013}. Finch’s paper \cite{Finch2004} and Wetzel’s paper \cite{Wetzel2003} have a separate section discussing this result and providing a detailed description.

The solution is applicable to many Lost-in-a-Forest shape configurations, particularly those with elongated shapes, (rectangles and isosceles triangles with small base angles) \cite{Finch2004}. Due to the duality between Bellman’s Lost-in-a-Forest Problem and Moser’s Worm Problem as Theorem 3 in Ref \cite{Finch2004}, this result also represents an open curve that can be covered by a unit strip \cite{Khandhawit2013} \cite{Wetzel2003} \cite{Adhikari1989}. However, in the existing literature, this shape has been derived exclusively through geometric methods. No analytical approach has been reported through a literature search.

In our previous papers \cite{Deng2024} \cite{Deng2026}, we proposed a general computational solution, together with proof for Bellman’s Lost-in-a-Forest Problem. The framework discretizes the problem as a Traveling Salesman Problem with Neighborhoods (TSPN). However, discretization introduces a large number of binary variables, making the resulting optimization problem computationally challenging. In this paper, we revisited the escape path for the infinite unit strip, reformulated the discretized problem in a continuous setting, and subsequently transformed it into a novel constrained functional minimization problem.

\section{Continuous transformation and solution by convex optimization}

Building on our previous general computational solution to Bellman’s Lost-in-a-Forest Problem \cite{Deng2024}, the discrete formulations for unit strip by TSPN can be written as
\begin{equation}
\begin{aligned}
\text{minimize} & \|p_{a_1}-O\| + \sum_{i=2}^{MN} \|p_{a_i}-p_{a_{i-1}}\|\\
\text{subject to:} & 
\left[
x_{a_h} \cos \frac{2\pi i}{N} + y_{a_h} \sin \frac{2\pi i}{N} + \frac{k-1}{M}
\right]
\cdot
\left[
x_{a_h} \cos \frac{2\pi i}{N} + y_{a_h} \sin \frac{2\pi i}{N} + \frac{k-1}{M} - 1
\right] = 0, \\
& \forall i \in \{0, 1, 2, \ldots, N-1\}, \, \forall k \in \{1, 2, \ldots, M\}, \, \exists h \in \{1, 2, \ldots, MN\}.
\end{aligned}
\end{equation}

This is a TSPN with $MN$ constraints, each involving two straight lines. If using Miller-Tucker-Zemlin (MTZ) formulation to solve, the discretization introduces a large number of binary variables, making the optimization computationally challenging. Hence, we will provide simpler continuous transformation.

We begin with the original discrete constraint. We have discrete path consist of points $\mathbf{p}_h = (x_{a_h}, y_{a_h})$ for $h \in \{1, 2, \ldots, MN\}$ and normal vector $\hat{\mathbf{n}}_i = (\cos \theta_i, \sin \theta_i)^\top$ where $\theta_i = \frac{2\pi i}{N}$. The scalar projection of point $\mathbf{p}_h$ onto the normal is $P(h, \theta_i) = \mathbf{p}_h \cdot \hat{\mathbf{n}}_i$. Then let $r_k = \frac{k-1}{M}$ uniformly samples in $[0, 1)$. The zero-product property of two parallel line (unit strip) in Eq (4) becomes:
\begin{equation}
P(h, \theta_i) = -r_k \quad or \quad P(h, \theta_i) = 1 - r_k
\end{equation}

As the limit $M,N \to \infty$. The discrete path converges to a continuous parameterized curve $\mathbf{r}(\tau) = (x(\tau), y(\tau))$ for $\tau \in [0, 2\pi]$. The discrete angles $\theta_i$ become a continuous variable $t \in [0, 2\pi]$, and the fractions $r_k$ become a continuous variable $r \in [0, 1]$.

\begin{theorem}
As $M, N \to \infty$, the discrete constraint necessitates that the continuous curve $\mathbf{r}(\tau)$ possesses a convex hull whose width in every direction is at least $1$, and which contains the origin $(0,0)$.
\end{theorem}

\begin{proof}
In the continuous limit, the discrete constraints $\forall i \in \{0, 1, 2, \ldots, N-1\}, \, \forall k \in \{1, 2, \ldots, M\}, \, \exists h \in \{1, 2, \ldots, MN\}$ form a mapping $[0, 2\pi] \times [0, 1] \to [0, 2\pi]$. The projection is $P_t(\tau) = \mathbf{r}(\tau) \cdot \hat{\mathbf{n}}(t)$. Then the continuous constraint of unit strip becomes
\begin{equation}
\label{eq:cont_logic}
\forall t \in [0, 2\pi], \, \forall r \in [0, 1], \, \exists \tau \in [0, 2\pi] \text{ s.t. } P_t(\tau) \in \{-r, 1-r\}
\end{equation}

As $\mathbf{r}(\tau)$ is a connected curve, its projection $P_t(\tau)$ for a fixed angle $t$ must map to a closed interval $I_t \subset \mathbb{R}$ by the Intermediate Value Theorem. Let this interval be defined by minimum and maximum
\begin{equation}
I_t = [m(t), M(t)] = \left[ \min_{\tau} P_t(\tau), \, \max_{\tau} P_t(\tau) \right]
\end{equation}
Eq \eqref{eq:cont_logic} states that for every $r \in [0, 1]$, the interval $I_t$ must contain either $-r$ or $1-r$. The sets of values required
\begin{align}
-r \in I_t \iff  r \in [-M(t), -m(t)] \\
1-r \in I_t \iff r \in [1-M(t), 1-m(t)]
\end{align}

For the condition to hold for \textit{all} $r \in [0, 1]$, the union of these two intervals must completely cover the unit interval $[0, 1]$, thus
\begin{equation}
\label{eq:union}
[0, 1] \subseteq [-M(t), -m(t)] \cup [1-M(t), 1-m(t)]
\end{equation}

Notice that second interval is exactly $+1$ shift of first interval. For the union of two closed intervals to form a continuous cover of $[0, 1]$ without any gaps, we have
\begin{equation}
-m(t) \ge 1 - M(t) \implies M(t) - m(t) \ge 1
\end{equation}
\begin{equation}
-M(t) \le 0 \implies M(t) \ge 0
\end{equation}
\begin{equation}
1-m(t) \ge 1 \implies m(t) \le 0
\end{equation}
These are the continuous constraints.
\end{proof}

To optimize the escape path, we solve the minimum length satisfying the bounds. 
\begin{corollary}
When the convex hull is parameterized by the normal angle $t \in [0, 2\pi]$ of its boundary, the derived topological bounds precisely yield the non-local antipodal constraints.
\end{corollary}

\begin{proof}
Let the convex hull boundary be parameterized such that $\mathbf{r}(t)$ is the point maximizing the projection in direction $t$. Therefore:
\begin{equation}
M(t) = \mathbf{r}(t) \cdot \hat{\mathbf{n}}(t)
\end{equation}
Similarly, the minimum projection in direction $t$ occurs at the exact opposite side of the convex hull, which is in direction $t+\pi$. Therefore:
\begin{equation}
m(t) = \mathbf{r}(t+\pi) \cdot \hat{\mathbf{n}}(t+\pi)
\end{equation}
Since $\hat{\mathbf{n}}(t+\pi) = -\hat{\mathbf{n}}(t)$, we have $m(t) = -\mathbf{r}(t+\pi) \cdot \hat{\mathbf{n}}(t)$. 

Substituting into Eq (11-13) we have
\begin{equation}
M(t) - m(t) \ge 1 \implies (\mathbf{r}(t) - \mathbf{r}(t+\pi)) \cdot \hat{\mathbf{n}}(t) \ge 1
\end{equation}
\begin{equation}
M(t) \ge 0 \implies \mathbf{r}(t) \cdot \hat{\mathbf{n}}(t) \ge 0
\end{equation}
\begin{equation}
m(t) \le 0 \implies \mathbf{r}(t+\pi) \cdot \hat{\mathbf{n}}(t+\pi) \ge 0
\end{equation}
The last one is equivalent to the second-to-last one over the domain $[0, 2\pi]$.

This concludes the rigorous mapping from the discrete constraints to the variational antipodal constraints.
\end{proof}

Therefore, the continuous formulation of escape path for unit strip is: 

\begin{equation}
\begin{aligned}
\text{minimize} & \|\mathbf{r}(0)\| + \int_{0}^{2\pi} \|\mathbf{r}'(t)\| \, dt\\
\text{subject to:} & 
(\mathbf{r}(t) - \mathbf{r}(t+\pi)) \cdot \hat{\mathbf{n}}(t) \ge 1, \quad \forall t \in [0, \pi)\\
& \mathbf{r}(t) \cdot \hat{\mathbf{n}}(t) \ge 0, \quad \forall t \in [0, 2\pi) 
\end{aligned}
\end{equation}
where $\hat{\mathbf{n}}(t) = (\cos t, \sin t)^\top$.

As the optimal path may contain non-smooth cusps at $\pi$, the direct application of Calculus of Variations and Euler-Lagrange (EL) equations for Eq (19) is difficult. Instead, we discretize the functional minimization using finite difference method.

We discretize the angular domain $[0, 2\pi]$ into an even number of $\hat{N}$ nodes. Let $\Delta t = \frac{2\pi}{\hat{N}}$, $t_i=\frac{2\pi i}{\hat{N}}$. The discrete finite difference optimization becomes
\begin{equation}
\begin{aligned}
\text{minimize} & \sqrt{x_1^2 + y_1^2} + \sum_{i=1}^{\hat{N}-1} \sqrt{(x_{i+1} - x_i)^2 + (y_{i+1} - y_i)^2}\\
\text{subject to:} & 
(x_i - x_{i + \hat{N}/2}) \cos t_i + (y_i - y_{i + \hat{N}/2}) \sin t_i \ge 1, \quad \forall i \in \{1, 2, \ldots, \hat{N}/2\} \\
& x_i \cos t_i + y_i \sin t_i \ge 0, \quad \forall i \in \{1, 2, \ldots, \hat{N}\}
\end{aligned}
\end{equation}

Based on the finite difference form above in Eq (20), it is convex optimization problem that we can solve. Mathematica 14.3 code for the solution is in Appendix 1. The resulting escape path for infinite unit strip is shown in Figure 1, the same as Zalgaller's \cite{Finch2004} \cite{Zalgaller2005}.

\begin{figure}
    \centering
    \includegraphics[width=0.5\linewidth]{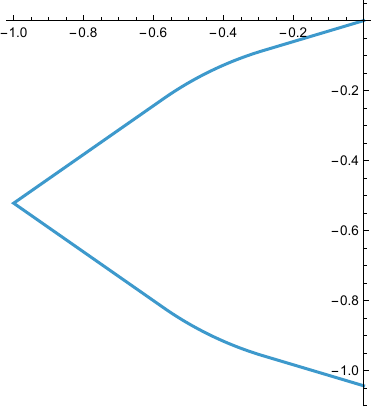}
    \caption{Resulting escape path for infinite unit strip}
\end{figure}

\section{Extend to $K$-Segment escape path}

Specifically in this section, we find the minimum length of polygonal curve with $K$ segments (defined by $K+1$ vertices). The geometric method was discussed in previous paper \cite{Finch2004-1}.

Let the curve be defined by an ordered set of vertices $\mathcal{P} = \{P_0=O, P_1, \dots, P_K\}$. And we use polar coordinates. The position $P_i$ for $i \in \{1, \dots, K\}$ is given by:
\begin{equation}
    P_i = (R_i \cos \phi_i, R_i \sin \phi_i)
\end{equation}

The objective is to minimize the total length $L$:
\begin{equation}
    \text{minimize}_{R, \phi} L = R_1 + \sum_{i=2}^{K} \sqrt{R_i^2 + R_{i-1}^2 - 2R_i R_{i-1} \cos(\phi_i - \phi_{i-1})}
\end{equation}

We utilize a rotational projection operator. The projection of vertex $P_i$ onto angle $t$ is:
\begin{equation}
    \text{proj}_i(t) = R_i \cos \phi_i \cos t + R_i \sin \phi_i \sin t = R_i \cos(t - \phi_i)
\end{equation}

The constraints in Eq (11) yield:
\begin{equation}
    \max_{i \in \{0, \dots, K\}} [\text{proj}_i(t)] - \min_{i \in \{0, \dots, K\}} [\text{proj}_i(t)] \ge 1, \quad \forall t \in [0, \pi)
\end{equation}

Figure 2 shows the results for K-segment cases with above objective function and constraints. Particularly, for the 2-segment case, the total length is $4\sqrt{3}/3$ and the angle between the two lines is $\pi/3$.

\begin{figure}
    \centering
    \includegraphics[width=0.49\linewidth]{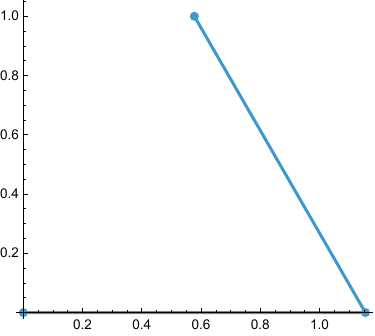}
    \includegraphics[width=0.49\linewidth]{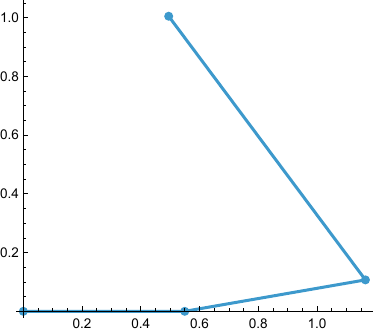}
    \includegraphics[width=0.49\linewidth]{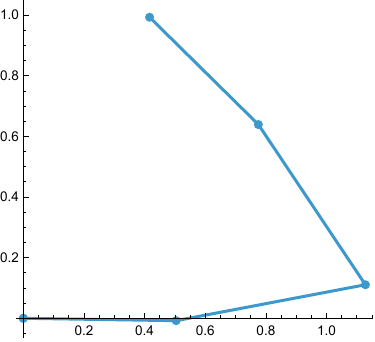}
    \caption{Results for 2, 3, 4-segment cases of escape path for infinite unit strip}
\end{figure}

\section{Extend to closed escape path}

In this section, we seek to find the closed escape path for unit strip. It is to find minimium of the following functional similar to Eq (19):

\begin{equation}
\begin{aligned}
\text{minimize} &\mathcal{J}[\mathbf{r}]= \|\mathbf{r}(0)\| + \int_{0}^{2\pi} \|\mathbf{r}'(t)\| \, dt+\|\mathbf{r}(2\pi)\| \\
\text{subject to:} & 
(\mathbf{r}(t) - \mathbf{r}(t+\pi)) \cdot \hat{\mathbf{n}}(t) \ge 1, \quad \forall t \in [0, \pi)\\
& \mathbf{r}(t) \cdot \hat{\mathbf{n}}(t) \ge 0, \quad \forall t \in [0, 2\pi) 
\end{aligned}
\end{equation}
We will prove that the result is a curve of constant unit width anchored at origin.

\begin{theorem}
Let $\mathbf{r}: [0, 2\pi] \to \mathbb{R}^2$ be a sufficiently smooth closed curve parameterized by $t$, with unit normal vector $\hat{\bm{n}}(t)$. The solution to the minimization above is a curve of constant unit width, achieving an absolute minimum functional value of $\pi$.
\end{theorem}

\begin{proof}
We introduce Lagrange multipliers $\lambda(t)$ for $t \in [0, \pi]$ and $\mu(t)$ for $t \in [0, 2\pi]$ to incorporate the constraints. The augmented functional is given by
\begin{equation}
    \mathcal{L}[\mathbf{r}] = \int_0^{2\pi} \|\mathbf{r}'(t)\| \, dt - \int_0^\pi \lambda(t) \big[ (\mathbf{r}(t) - \mathbf{r}(t+\pi)) \cdot \hat{\bm{n}}(t) - 1 \big] dt - \int_0^{2\pi} \mu(t) \mathbf{r}(t) \cdot \hat{\bm{n}}(t) \, dt.
\end{equation}

By introducing the variation $\mathbf{r}(t) \to \mathbf{r}(t) + \epsilon \bm{\eta}(t)$, we obtain the EL equation:
\begin{equation}
    \frac{d\bm{T}}{dt} = - \lambda(t) \hat{\mathbf{n}}(t)
\end{equation}
where $\bm{T}(t) = \frac{\mathbf{r}'(t)}{\|\mathbf{r}'(t)\|}$ is the unit tangent vector. The alignment of $\frac{d\bm{T}}{dt}$ with $\hat{\mathbf{n}}(t) = (\cos t, \sin t)$ rigorously forces $\bm{T}(t) = (-\sin t, \cos t)$. This establishes that the parameter $t$ is exactly the angle of the outward normal, confirming the curve is convex. Therefore, we have the support function $p(t) = \mathbf{r}(t) \cdot \hat{\mathbf{n}}(t)$

The geometric width of the curve in direction $\hat{\mathbf{n}}(t)$ is simply the sum of opposite support functions $w(t) = p(t) + p(t+\pi)$. Thus, the constraints translate exactly to
\begin{equation}
    w(t) \ge 1 \quad \text{and} \quad p(t) \ge 0
\end{equation}

As the radius of curvature is $p(t) + p''(t)$, the total length of the closed curve is 
\begin{equation}
    L=\int_0^{2\pi} (p(t) + p''(t)) \, dt
\end{equation}
As $p(t)$ is periodic for a closed curve, the integral of $p''(t)$ evaluates to zero. We decompose the remaining integral over half-periods:
\begin{equation}
    L = \int_0^{2\pi} p(t) \, dt = \int_0^\pi p(t) \, dt + \int_0^\pi p(t+\pi) \, dt = \int_0^\pi w(t) \, dt
\end{equation}
Since the constraint mandates $w(t) \ge 1$, the integral is strictly minimized when $w(t) = 1, \forall t \in [0, \pi]$. This condition perfectly defines a curve of constant unit width. By Barbier's Theorem, the perimeter is exactly $\pi$, which is the minimum. 

Using the support function, the position vector is $\mathbf{r}(t) = p(t)\hat{\mathbf{n}}(t) + p'(t)\bm{T}(t)$. Substituting minimized length into the original functional yields:
\begin{equation}
    \|\mathbf{r}(0)\|+ \pi +\|\mathbf{r}(2\pi)\|
\end{equation}
This theoretical limit is achieved if and only if both boundary terms independently vanish:
\begin{equation}
    p(0) = 0 \quad \text{and} \quad p'(0) = 0
\end{equation}
With periodicity, $p(2\pi) = p'(2\pi) = 0$, guaranteeing that $\|\mathbf{r}(2\pi)\| = 0$.

\end{proof}

\section{Appendix-Mathematica notebook}
The appendix provide Mathematica notebook that contain detailed convex optimization and results presented in Section 3.

~\\
College of Engineering and Computer Science, University of Central Florida, Orlando, FL, USA

Email: \underline{zhipeng.deng@ucf.edu}

\includepdf[pages=-]{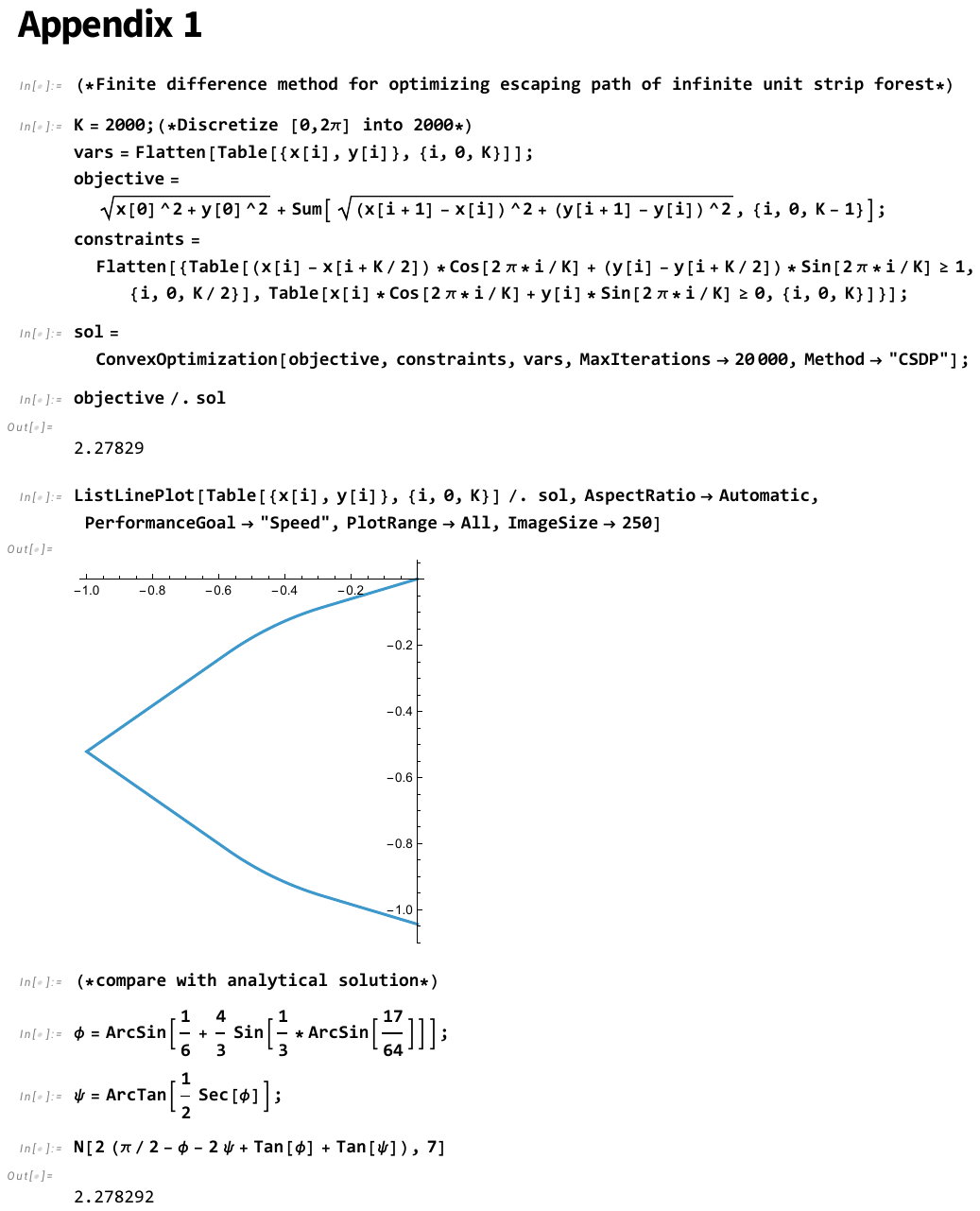}

\end{document}